\documentclass[12pt]{article}

\usepackage{amsmath,amssymb}

\newtheorem{theorem}{Theorem}[section]
\newtheorem{proposition}[theorem]{Proposition}
\newtheorem{assumption}[theorem]{Assumption}

\newenvironment{proof}{\smallskip\par{\sc Proof.}\enspace}%
 {{\unskip\nobreak\hfil\penalty50\hskip2em
          \hbox{}\nobreak\hfil{\rule[-1pt]{5pt}{10pt}}
          \parfillskip=0pt\finalhyphendemerits=0
          \par\medskip}} 

\makeatletter
\def\section{\@startsection {section}{1}{\z@}{3.25ex plus 1ex minus
 .2ex}{1.5ex plus .2ex}{\large\bf}}
\def\subsection{\@startsection{subsection}{2}{\z@}{3.25ex plus 1ex minus
 .2ex}{1.5ex plus .2ex}{\normalsize\bf}}
\@addtoreset{equation}{section} 
\makeatother

\title{ Prohorov-type local limit theorems on\\ abstract Wiener spaces }

\author{Alberto Lanconelli\footnote{Dipartimento di Matematica, Universit\'a degli Studi di Bari Aldo Moro, Via E. Orabona 4, 70125 Bari - Italia. E-mail: \emph{alberto.lanconelli@uniba.it}}}

\date{\empty}

\begin{document}

\maketitle

\numberwithin{equation}{section}

\bigskip

\begin{abstract}
We prove that the density of $\frac{X_1+\cdot\cdot\cdot+X_n-nE[X_1]}{\sqrt{n}}$, where $\{X_n\}_{n\geq 1}$ is a sequence of independent and identically distributed random variables taking values on a abstract Wiener space, converges in $\mathcal{L}^1$ to the density of a certain Gaussian measure which is absolutely continuous with respect to the reference Wiener measure. The crucial feature in our investigation is that we do not require the covariance structure of $\{X_n\}_{n\geq 1}$ to coincide with the one of the Wiener measure. This produces a non trivial (different from the constant function one) limiting object which reflects the different covariance structures involved. The present paper generalizes the results proved in \cite{LS 2016} and deepens the connection between local limit theorems on (infinite dimensional) Gaussian spaces and some key tools from the Analysis on the Wiener space, like the Wiener-It\^o chaos decomposition, Ornstein-Uhlenbeck semigroup and Wick product. We also verify and discuss our main assumptions on some examples arising from the applications: dimension independent Berry-Esseen-type bounds and weak solutions of stochastic differential equations.
\end{abstract}



\allowdisplaybreaks

\section{Introduction and statement of the main result}

Local limit theorems are central limit theorems for densities. Given a sequence $\{X_n\}_{n\geq 1}$ of independent and identically distributed random variables, one aims to prove that the density of the standardized sum $\frac{X_1+\cdot\cdot\cdot+X_n-nE[X_1]}{\sqrt{nVar(X_1)}}$ converges in some sense to the standard normal density function. This problem has attracted the attention of several authors: we recall the classical papers by Prohorov \cite{Prohorov}, which is concerned with convergence in $\mathcal{L}^1$, Gnedenko \cite{Gnedenko}, who studies uniform convergence and Ranga Rao and Varadarajan \cite{Ranga Rao Varadarajan}, where point-wise convergence is investigated. All these three classical results rely on Fourier transform techniques applied to convolutions of densities. More recently, Barron \cite{Barron} proved that the relative entropy, also known as Kullback-Leibler divergence, of the law of $\frac{X_1+\cdot\cdot\cdot+X_n-nE[X_1]}{\sqrt{nVar(X_1)}}$ with respect to the standard Gaussian measure tends to zero, monotonically along a certain subsequence. This type of convergence is stronger than the one considered by Prohorov \cite{Prohorov} and hence improves his result. \\
When the sequence $\{X_n\}_{n\geq 1}$ takes values on an infinite dimensional space, then the validity of general local limit theorems is not guaranteed. In fact, Bloznelis \cite{Blozenis} has shown a counterexample to the validity of Prohorov's theorem on general Hilbert spaces. We also mention the paper by Davydov \cite{Davydov}, where a variant of a local limit theorem for Banach space valued random variables is proposed.\\
The aim of the present paper is to prove a version of Prohorov's theorem for an abstract Wiener space valued random sequence. More precisely, we show that the density of $\frac{X_1+\cdot\cdot\cdot+X_n-nE[X_1]}{\sqrt{n}}$, where $\{X_n\}_{n\geq 1}$ is a sequence of independent and identically distributed random variables taking values on an abstract Wiener space, converges in $\mathcal{L}^1$ to the density of a certain Gaussian measure which is absolutely continuous with respect to the reference Wiener measure. The choice of working in a space endowed with a Gaussian measure, which is the limiting object of the central/local limit theorem, brings important advantages: the roles of scaling operator and convolution product are naturally played by the Ornstein-Uhlenbeck semigroup and Wick product, respectively (see Theorem \ref{wick product convolution} below). This fact together with certain norm inequalities proved in previous papers (see Theorem \ref{Young inequality} below) become the key ingredients for the proof of our local limit theorem. Furthermore, we do not require the covariance operator of the $X_n$'s to be the identity, which is the covariance of the Wiener measure; we set three standing assumptions which allow for a richer family of Gaussian measures as limiting objects.\\

\noindent To be more specific, we recall that $(H,W,\mu)$ is an \emph{abstract Wiener space} if
$(H,\langle\cdot,\cdot\rangle_H)$ is a separable Hilbert space, which
is continuously and densely embedded in the Banach space
$(W,|\cdot|_W)$, and $\mu$ is a Gaussian probability measure
on the Borel sets of $W$ such that
\begin{eqnarray}\label{Gaussian characteristic}
\int_{W}e^{i\langle w,\varphi\rangle}d\mu(w)=e^{-\frac{1}{2}|\varphi|_H^2},\quad\mbox{ for all }\varphi\in W^*.
\end{eqnarray}
Here $W^*\subset H$ denotes the dual space of $W$, which in turn is
dense in $H$, and $\langle\cdot,\cdot\rangle$ stands for the dual
pairing between $W$ and $W^*$. We will refer to $H$ as the
\emph{Cameron-Martin} space of $W$. Observe that $\langle w,\varphi\rangle=\langle w,\varphi\rangle_H$ if $w\in H$ and $\varphi\in W^*$. Among the most important examples of abstract Wiener spaces, we recall: the Euclidean space $\mathbb{R}^d$ endowed with the standard $d$-dimensional Gaussian measure (in this case we have $W=H=\mathbb{R}^d$); the classical Wiener space $C_0([0,1])$ of continuous functions starting at zero endowed with the classical Wiener measure (in this case $W=C_0([0,1])$ and $H=H_0^1([0,1])$). In the sequel we denote by $\Vert\cdot\Vert_p$ the norm in the space $\mathcal{L}^p(W,\mu)$ for $p\geq 1$.\\
We now describe our set of assumptions. Let $\{X_n\}_{n\geq 1}$ be a sequence of independent and identically distributed random variables taking values on $W$.

\begin{assumption}\label{square integrable density}
The law of the $X_n$'s is absolutely continuous with respect to $\mu$ with a density $f$ belonging to $\mathcal{L}^2(W,\mu)$
\end{assumption}

\noindent Assumption \ref{square integrable density} has several important implications. First of all, it yields the finiteness of all the moments of the scalar random variable $\langle X_n,\varphi\rangle$ for $\varphi\in W^*$. In fact, for any $m\in\mathbb{N}$ and $\varphi\in W^*$ a simple application of the Cauchy-Schwartz inequality gives
\begin{eqnarray*}
E[|\langle X_n,\varphi\rangle|^m]&=&\int_W|\langle w,\varphi\rangle|^md\mu_{X_n}(w)\\
&=&\int_W|\langle w,\varphi\rangle|^m f(w)d\mu(w)\\
&\leq&\Big(\int_W|\langle w,\varphi\rangle|^{2m} d\mu(w)\Big)^{\frac{1}{2}}\cdot\Big(\int_W f^2(w)d\mu(w)\Big)^{\frac{1}{2}}\\
&=&\sqrt{(2m-1)!!}|\varphi|_H^{m}\Big(\int_W f^2(w)d\mu(w)\Big)^{\frac{1}{2}}
\end{eqnarray*}
entailing the finiteness of the moments. Moreover, taking $m=2$ in the previous inequality, we see that if $\{\varphi_j\}_{j\geq 1}$ is a sequence in $W^*$ converging in the norm of $H$ to $h\in H$, then
\begin{eqnarray*}
E[|\langle X_n,\varphi_j\rangle-\langle X_n,\varphi_i\rangle|^2]&=&E[|\langle X_n,\varphi_j-\varphi_i\rangle|^2]\\
&\leq&\sqrt{3}|\varphi_j-\varphi_i|_H^2\Big(\int_W f^2(w)d\mu(w)\Big)^{\frac{1}{2}}.
\end{eqnarray*}
Therefore, $\{\langle X_n,\varphi_j\rangle\}_{j\geq 1}$ turns out to be a Cauchy sequence in $\mathcal{L}^2(W,\mu)$ and we can define $\langle X_n,h\rangle$ almost surely as the limit of this sequence.\\
Another consequence of Assumption \ref{square integrable density} is that, according to the Wiener-It\^o chaos decomposition theorem, the density $f$ can be represented as $\sum_{k\geq 0}\delta^k(f_k)$, where
\begin{eqnarray*}
\delta^0(f_0)=\int_Wf(w)d\mu(w)=1
\end{eqnarray*}
and for $k\geq 1$, $f_k\in H^{\hat{\otimes}k}$, the space of symmetric elements of $H^{\otimes k}$, while $\delta^k(f_k)$ stands for the multiple Wiener-It\^o integral of $f_n$. Hence, we can write
\begin{eqnarray*}
Var(\langle X_n, h\rangle)&=&E[\langle X_n, h\rangle^2]-E[\langle X_n, h\rangle]^2\nonumber\\
&=&\int_W\langle w, h\rangle^2f(w)d\mu(w)-\Big(\int_W\langle w, h\rangle f(w)d\mu(w)\Big)^2\nonumber\\
&=&\int_W(\delta^2(h^{\otimes 2})+|h|_H^2)f(w)d\mu(w)-\Big(\int_W\langle w, h\rangle f(w)d\mu(w)\Big)^2\nonumber\\
&=&\int_W\delta^2(h^{\otimes 2})f(w)d\mu(w)+|h|_H^2-\Big(\int_W\langle w, h\rangle f(w)d\mu(w)\Big)^2\nonumber\\
&=&2\langle f_2,h^{\otimes 2}\rangle_{H^{\otimes 2}}-\langle f_1,h\rangle_H^2+|h|_H^2\nonumber\\
&=&\langle 2f_2-f_1^{\otimes 2},h^{\otimes 2}\rangle_{H^{\otimes 2}}+|h|_H^2.
\end{eqnarray*}
This gives the identity
\begin{eqnarray}\label{variance chaos expansion}
Var(\langle X_n, h\rangle)&=&\langle 2f_2-f_1^{\otimes 2},h^{\otimes 2}\rangle_{H^{\otimes 2}}+|h|_H^2
\end{eqnarray}
that will play a crucial role in the sequel. Our next assumption concerns the behaviour of $Var(\langle X_n,h\rangle)$ for $h\in H$.

\begin{assumption}\label{positive trace class}
For every $h\in H$,
\begin{eqnarray}\label{positive}
Var(\langle X_n,h\rangle)\geq |h|_H^2
\end{eqnarray}
and for some orthonormal basis $\{e_j\}_{j\geq 1}$ of $H$,
\begin{eqnarray}\label{trace class}
\sum_{j\geq 1}[Var(\langle X_n,e_j\rangle)-1]<+\infty.
\end{eqnarray}
\end{assumption}

\noindent Observe that, according to equation (\ref{variance chaos expansion}), inequality (\ref{positive}) is equivalent to the positivity of $2f_2-f_1^{\otimes 2}$; in fact for all $h\in H$ we have
\begin{eqnarray*}
\langle 2f_2-f_1^{\otimes 2},h^{\otimes 2}\rangle_{H^{\otimes 2}}&=&Var(\langle X_n, h\rangle)-|h|_H^2.
\end{eqnarray*}
Moreover, inequality (\ref{trace class}) ensures that $2f_2-f_1^{\otimes 2}$ is of trace class since
\begin{eqnarray*}
\sum_{j\geq 1}\langle 2f_2-f_1^{\otimes 2},e_j^{\otimes 2}\rangle_{H^{\otimes 2}}&=&\sum_{j\geq 1}[Var(\langle X_n, e_j\rangle)-1].
\end{eqnarray*}

\noindent Our last assumption is the following.

\begin{assumption}\label{integrable quadratic exponential}
For some orthonormal basis $\{e_j\}_{j\geq 1}$ of $H$,
\begin{eqnarray}\label{series covariance}
\sum_{i,j\geq 1}[Cov(\langle X_n,e_i\rangle,\langle X_n,e_j\rangle)-\delta_{ij}]^2<1.
\end{eqnarray}
\end{assumption}

\noindent This is a more technical condition: it serves to guarantee that the limiting Gaussian measure in our local limit theorem possesses a square integrable density with respect to the reference Wiener measure (see Proposition \ref{properties quadratic exp} below). Note that the left hand side in (\ref{series covariance}) corresponds to $|2f_2-f_1^{\otimes 2}|^2_{H^{\otimes 2}}$; in fact,
\begin{eqnarray*}
|2f_2-f_1^{\otimes 2}|^2_{H^{\otimes 2}}&=&\sum_{i,j\geq 1}\langle 2f_2-f_1^{\otimes 2},e_i\otimes e_j\rangle_{H^{\otimes 2}}^2\\
&=&\sum_{i,j\geq 1}\langle 2f_2-f_1^{\otimes 2},e_i\hat{\otimes}e_j\rangle_{H^{\otimes 2}}^2\\
&=&\sum_{i,j\geq 1}\Big[\int_W\delta^2(e_i\hat{\otimes}e_j)f(w)d\mu(w)\\
&&-\int_W\langle w,e_i\rangle f(w)d\mu(w)\cdot\int_W\langle w,e_j\rangle f(w)d\mu(w)\Big]^2\\
&=&\sum_{i,j\geq 1}\Big[\int_W\langle w,e_i\rangle\cdot\langle w,e_j\rangle f(w)d\mu(w)-\langle e_i,e_j\rangle_H\\
&&-\int_W\langle w,e_i\rangle f(w)d\mu(w)\cdot\int_W\langle w,e_j\rangle f(w)d\mu(w)\Big]^2\\
&=&\sum_{i,j\geq 1}[Cov(\langle X_n,e_i\rangle,\langle X_n,e_j\rangle)-\delta_{ij}]^2.
\end{eqnarray*}
Hence, Assumption \ref{integrable quadratic exponential} is equivalent to
\begin{eqnarray*}
\Big|f_2-\frac{f_1^{\otimes 2}}{2}\Big|^2_{H^{\otimes 2}}<\frac{1}{4}.
\end{eqnarray*}
The reason for this assumption comes from the advantages of utilizing the Wiener-It\^o chaos expansion and the hyper-contractive properties of the Ornstein-Uhlenbeck semigroup (see Nelson \cite{Nelson}), which are absent in the $\mathcal{L}^1(W,\mu)$ space.\\
The next proposition is a first step towards the main result. Its proof is deferred to Section 3.

\begin{proposition}\label{properties quadratic exp}
Let Assumptions \ref{square integrable density}, \ref{positive trace class} and \ref{integrable quadratic exponential} be in force and set $g_2:=f_2-\frac{f_1^{\otimes 2}}{2}$, where $f_1$ and $f_2$ are the first and second order kernels in the chaos decomposition of $f$, respectively. Then,
\begin{eqnarray*}
\xi=\sum_{k\geq 0}\frac{\delta^{2k}(g_2^{\otimes k})}{k!}
\end{eqnarray*}
is the square integrable density of a Gaussian measure on $W$ with characteristic functional given by
\begin{eqnarray*}
\exp\Big\{-\langle g_2,h^{\otimes 2}\rangle-\frac{\Vert h\Vert_H^2}{2}\Big\}.
\end{eqnarray*}
Moreover, if for any $n\in\mathbb{N}$ the random variables $\mathcal{X}_1,...,\mathcal{X}_n$ are independent with law $\xi d\mu$, then
\begin{eqnarray}\label{WID}
\frac{\mathcal{X}_1+\cdot\cdot\cdot+\mathcal{X}_n}{\sqrt{n}}\quad\mbox{ has law }\quad\xi d\mu.
\end{eqnarray}
\end{proposition}
We will refer to the Gaussian measure $\xi d\mu$ on $W$ as the \emph{Gaussian measure induced by} $f$. We are now ready to state the main result of the present paper. The proof is postponed to Section 4.

\begin{theorem}\label{main theorem}
Let $\{X_n\}_{n\geq 1}$ be a sequence of independent and identically distributed random variables taking values on $W$ and fulfilling Assumptions \ref{square integrable density}, \ref{positive trace class} and \ref{integrable quadratic exponential}. Let also $\mathcal{X}$ be a random variable on $W$ whose law is the Gaussian measure induced by $f$. Then, for any $\alpha\in ]0,1[$ the density of
\begin{eqnarray}\label{limit}
\sqrt{\alpha}\cdot\frac{X_1+\cdot\cdot\cdot+X_n-nE[X_1]}{\sqrt{n}}+\sqrt{1-\alpha}\cdot Z
\end{eqnarray}
converges in $\mathcal{L}^1(W,\mu)$ as $n$ tends to infinity to the density of
\begin{eqnarray}\label{limit2}
\sqrt{\alpha}\cdot\mathcal{X}+\sqrt{1-\alpha}\cdot Z
\end{eqnarray}
with speed of convergence of order $\frac{1}{\sqrt{n}}$. Here $Z$ is a $W$-valued random variable with law $\mu$ which is independent of the sequence $\{X_n\}_{n\geq 1}$ and of $\mathcal{X}$.
\end{theorem}

We note that if the covariance structure of the variables $X_n$'s coincides with the one of the Wiener measure $\mu$, i.e. $Var(\langle X,h\rangle)=|h|_H^2$, then Assumptions \ref{positive trace class} and \ref{integrable quadratic exponential} are trivially satisfied. This corresponds to the classical assumption in multidimensional central/local limit theorems. Moreover, Theorem \ref{main theorem} generalizes the result proved in \cite{LS 2016}, where it is assumed that $f_1=0$ and $f_2=0$ (from a probabilistic point of view, this is equivalent to the requirements $E[X_n]=0$ and $Var(\langle X,h\rangle)=|h|_H^2$).\\
The reason why the local limit theorem we prove concerns the sequence in (\ref{limit}) instead of $\frac{X_1+\cdot\cdot\cdot+X_n-nE[X_1]}{\sqrt{n}}$, which would be the natural object of investigation, is of technical nature. In fact, observe first of all that adding the term containing $Z$ has the effect of smoothing the density of the sequence $X_n$ through the Ornstein-Uhlenbeck semigroup. Moreover, according to Theorem \ref{wick product convolution} below the Wick product behaves as a Gaussian convolution, i.e. is the necessary tool to handle densities of sum of independent random variables on Gaussian spaces. The Wick product is however an unbounded bilinear form on any $\mathcal{L}^p(W,\mu)$ space; to get something in $\mathcal{L}^p(W,\mu)$ out of a Wick product, one has to smooth the densities involved through the Ornstein-Uhlenbeck semigroup. This is contained in Theorem \ref{Young inequality} below. The point here is that the regularization required by that theorem exhausts all the infinitesimal scaling effect of the factor $\frac{1}{\sqrt{n}}$ in $\frac{X_1+\cdot\cdot\cdot+X_n-nE[X_1]}{\sqrt{n}}$. Hence, we are forced to add an extra smoothing component on the density of the $X_n$'s in order to prove the theorem. We remark however that such a smoothness assumption is also required by Lynnik \cite{Lynnik} in proving an information-theoretic central limit theorem.\\

The paper is organized as follows: in Section 2 the crucial assumptions of the main theorem, Theorem \ref{main theorem}, are verified and discussed on three important examples arising from the applications: the finite dimensional case in relation to dimension independent Berry-Esseen-type bounds, the class of measures of convolution type obtained by convolving the Wiener measure with a probability measure supported on the Cameron-Martin space, the analysis of weak solutions of a family of systems of stochastic differential equations; Section 3 contains the main preparatory theorems needed for the proof of the main result: the roles of the Wick product and Ornstein-Uhlenbeck semigroup in the manipulation of probability densities on infinite dimensional Gaussian spaces, Theorem \ref{wick product convolution} and Theorem \ref{Young inequality}, and the proof of Proposition \ref{properties quadratic exp}, which describes the properties of the limiting object of our local limit theorem; in Section 4 we prove Theorem \ref{main theorem} while in the Appendix we collect for the reader's convenience several useful formulas utilized throughout the paper.

\section{Examples and applications}

In this section we check and discuss the assumptions of Theorem \ref{main theorem} in some concrete examples arising from the applications.

\subsection{The finite dimensional case: dimension independent Berry-Esseen-type bounds}

We choose $W=H=\mathbb{R}^d$ and
\begin{eqnarray*}
\mu_d(A)=\int_{A}(2\pi)^{-\frac{d}{2}}e^{-\frac{|x|^2}{2}}dx,\quad A\in\mathcal{B}(\mathbb{R}^d)
\end{eqnarray*}
where $|\cdot|$ denotes the $d$-dimensional Euclidean norm and $\mathcal{B}(\mathbb{R}^d)$ is the collection of the Borel sets of $\mathbb{R}^d$. Then, $(W,H,\mu)$ is a finite dimensional abstract Wiener space. \\
Let $\{X_n\}_{n\geq 1}$ be a sequence of independent and identically distributed $d$-dimensional random vectors; assume that the law of $X_n$ on $\mathbb{R}^d$ is absolutely continuous with respect to $\mu_d$ with a density $f\in\mathcal{L}^2(\mathbb{R}^d,\mu_d)$ (this corresponds to Assumption \ref{square integrable density}). Moreover, observe that Assumptions \ref{positive trace class} and \ref{integrable quadratic exponential} can be easily rephrased in terms of the covariance matrix of the random vector $X_n$ (condition (\ref{trace class}) is always satisfied in the finite dimensional framework). If we denote by $\mu_n$ the law of
\begin{eqnarray*}
\sqrt{\alpha}\cdot\frac{X_1+\cdot\cdot\cdot+X_n-nE[X_1]}{\sqrt{n}}+\sqrt{1-\alpha}\cdot Z
\end{eqnarray*}
and by $\mu_{\mathcal{X}}$ the law of $\sqrt{\alpha}\cdot\mathcal{X}+\sqrt{1-\alpha}\cdot Z$, then Theorem \ref{main theorem} can rewritten as
\begin{eqnarray}\label{berry-esseen}
d_{TV}(\mu_n,\mu_{\mathcal{X}})\leq \frac{C}{\sqrt{n}}
\end{eqnarray}
where $d_{TV}$ stands for the distance in total variation and $C$ is a constant depending only on $f$ and $\alpha$. Inequality (\ref{berry-esseen}) represents a dimension independent Berry-Essen type bound. These type of estimates have been investigated by many authors under very mild conditions (existence of the third moment) and with a constant $C$ in (\ref{berry-esseen}) depending on the dimension $d$. See Bentkus \cite{Bentkus} for the best known value of $C$ and the references quoted there. We stress that in the aforementioned paper it is assumed that the covariance of the vector $X_n$ is the identity matrix; if we make the same assumption, then Assumptions \ref{positive trace class} and \ref{integrable quadratic exponential} are trivially satisfied. Therefore, compared to the paper \cite{Bentkus}, the additional regularity we impose to the $X_n$'s is Assumption \ref{square integrable density} together with the introduction of the smoothing parameter $\alpha\in ]0,1[$. To conclude, our approach needs more stringent assumptions on the law of the random vector $X_n$ but provides a bound which does not depend on the dimension of the image space of the random sequence $\{X_n\}_{n\geq 1}$ (see \cite{LS 2016} for a more details on this type of comparison).

\subsection{The cases $f_2-\frac{f_1^{\otimes 2}}{2}=0$ and $f_2-\frac{f_1^{\otimes 2}}{2}=g^{\otimes 2}$}

To ease the notation set $g_2:=f_2-\frac{f_1^{\otimes 2}}{2}$ and assume first that $g_2=0$. This means that according to (\ref{variance chaos expansion}) we have
\begin{eqnarray*}
Var(\langle X_n,h\rangle)=|h|^2_H\quad h\in H.
\end{eqnarray*}
Hence, Assumptions \ref{positive trace class} and \ref{integrable quadratic exponential} are trivially fulfilled. Moreover, the density $\xi$ in Proposition \ref{properties quadratic exp} reduces to the constant function one and we recover a slight generalization of the result proved in \cite{LS 2016} (where it is assumed that $f_1=0$ and $f_2=0$).\\

\noindent Now assume that $g_2=g^{\otimes 2}$ for some $g\in H$. Then, $g_2$ is a positive trace class element of $H^{\hat{\otimes} 2}$, which is equivalent to Assumption \ref{positive trace class}. Furthermore, if $|g|^2_H<\frac{1}{2}$, then also Assumption \ref{integrable quadratic exponential} is satisfied. In this case the density $\xi$ from Proposition \ref{properties quadratic exp} looks like
\begin{eqnarray*}
\sum_{k\geq 0}\frac{\delta^{2k}(g^{\otimes 2k})}{k!}.
\end{eqnarray*}
The last expression can be written in the formalism of the Wick Calculus (e.g. Holden et al. \cite{HOUZ}) as
\begin{eqnarray*}
\exp^{\diamond}\{\delta(g)^{\diamond 2}\}\quad\mbox{ or alternatively }\quad :\exp\{\delta(g)^2\}:\quad.
\end{eqnarray*}
It is proved in Aase et al. \cite{AOU} that
\begin{eqnarray}\label{oksendal}
\exp^{\diamond}\{\delta(g)^{\diamond 2}\}=\frac{1}{\sqrt{1+2 |g|_H^2}}\exp\Big\{\frac{\delta(g)^2}{1+2 |g|_H^2}\Big\},
\end{eqnarray}
provided that $2|g|_H^2<1$. If we set $\tilde{g}=\frac{g}{|g|_H}$, then equation (\ref{oksendal}) becomes
\begin{eqnarray*}
\exp^{\diamond}\{\delta(g)^{\diamond 2}\}&=&\frac{1}{\sqrt{1+2 |g|_H^2}}\exp\Big\{\frac{\delta(g)^2}{1+2 |g|_H^2}\Big\}\\
&=&\frac{1}{\sqrt{1+2 |g|_H^2}}\exp\Big\{\frac{2 |g|_H^2}{1+2 |g|_H^2}\frac{\delta(\tilde{g})^2}{2}\Big\}\\
&=&\frac{1}{\sqrt{1+2 |g|_H^2}}\exp\Big\{\Big(1-\frac{1}{1+2 |g|_H^2}\Big)\frac{\delta(\tilde{g})^2}{2}\Big\}\\
&=&\frac{\frac{1}{\sqrt{2\pi(1+2 |g|_H^2)}}e^{-\frac{x^2}{2(1+2 |g|_H^2)}}}{\frac{1}{\sqrt{2\pi}}e^{-\frac{x^2}{2}}}\Big\vert_{x=\delta(\tilde{g})}.
\end{eqnarray*}
Hence, the density $\xi$ can be written more explicitly as
\begin{eqnarray*}
\sum_{k\geq 0}\frac{\delta^{2k}(g^{\otimes 2k})}{k!}=\frac{\frac{1}{\sqrt{2\pi(1+2 |g|_H^2)}}e^{-\frac{x^2}{2(1+2 |g|_H^2)}}}{\frac{1}{\sqrt{2\pi}}e^{-\frac{x^2}{2}}}\Big\vert_{x=\delta(\tilde{\tilde{g}})}.
\end{eqnarray*}

\subsection{Convolution measures}

Let $(W,H,\mu)$ be an abstract Wiener space and let $X$ be a random variable taking values on $W$. Assume that $X=Z+Y$ where $Z$ and $Y$ are independent, the law of $Z$ is $\mu$ and the law of $Y$, say $\nu$, is supported on $H$. Then, the law of $X$ is given by $\mu\star\nu$ where $\star$ denotes the convolution of probability measures. This class of probability measures has an important role in the applications being a Gaussian (white noise) perturbation of a probability measure on the Hilbert space $H$. Poincar\'e-type inequalities with respect to this class of measures have been investigated in \cite{L 2016} and \cite{DLS 2016}.\\
Observe that the measure $\mu\star\nu$ is absolutely continuous with respect to $\mu$ with a density given by
\begin{eqnarray}\label{density convolution}
\frac{d (\mu\star\nu)}{d \mu}=\int_H\mathcal{E}(h)d\nu(h)
\end{eqnarray}
(here $\mathcal{E}(h)$ denotes the stochastic exponential: see (\ref{def stochastic exponential}) in the Appendix below). We now want to check that the measure $\mu\star\nu$ fulfills the assumptions of Theorem \ref{main theorem} . First of all, we need to verify the membership of (\ref{density convolution}) to $\mathcal{L}^2(W,\mu)$. According to the Minkowsky integral inequality we can write
\begin{eqnarray*}
\Big\Vert\int_H\mathcal{E}(h)d\nu(h)\Big\Vert_2&\leq&\int_H\Vert\mathcal{E}(h)\Vert_2d\nu(h)\\
&=&\int_H\exp\Big\{\frac{|h|_H^2}{2}\Big\}d\nu(h).
\end{eqnarray*}
Therefore, the membership of (\ref{density convolution}) to $\mathcal{L}^2(W,\mu)$ is guaranteed if $\nu$ satisfies the following exponential integrability condition
\begin{eqnarray}\label{exponential integrability}
\int_H\exp\Big\{\frac{|h|_H^2}{2}\Big\}d\nu(h)<+\infty.
\end{eqnarray}
We now compute the variance of $\langle X,h\rangle$ for $h\in H$. We have:
\begin{eqnarray*}
Var(\langle X,h\rangle)&=&Var(\langle Z+Y,h\rangle)\\
&=&Var(\langle Z,h\rangle+\langle Y,h\rangle)\\
&=&Var(\langle Z,h\rangle)+Var(\langle Y,h\rangle)\\
&=&|h|_H^2+Var(\langle Y,h\rangle).
\end{eqnarray*}
This yields immediately (\ref{positive}). In addition,
\begin{eqnarray*}
\sum_{j\geq 1}[Var(\langle X,e_j\rangle)-1]=\sum_{j\geq 1}Var(\langle Y,e_j\rangle).
\end{eqnarray*}
Since the measure $\nu$, the law of $Y$, is supported on the Hilbert space $H$ and  satisfies the condition (\ref{exponential integrability}), it follows from the previous equality that also (\ref{trace class}) from Assumption \ref{positive} is satisfied. We now verify the last assumption; by construction (the independence of $Z$ and $Y$) we can write
\begin{eqnarray*}
&&\sum_{i,j\geq 1}[Cov(\langle X,e_i\rangle,\langle X,e_j\rangle)-\delta_{ij}]^2\\
&=&\sum_{i,j\geq 1}[Cov(\langle Z,e_i\rangle+\langle Y,e_i\rangle,\langle Z,e_j\rangle+\langle Y,e_j\rangle)-\delta_{ij}]^2\\
&=&\sum_{i,j\geq 1}[Cov(\langle Z,e_i\rangle,\langle Z,e_j\rangle)+Cov(\langle Y,e_i\rangle,\langle Y,e_j\rangle)-\delta_{ij}]^2\\
&=&\sum_{i,j\geq 1}Cov(\langle Y,e_i\rangle,\langle Y,e_j\rangle)^2\\
&\leq&\Big(\sum_{i\geq 1}Var(\langle Y,e_i\rangle)\Big)^2.
\end{eqnarray*}
Hence, Assumption \ref{integrable quadratic exponential} is fulfilled if, for instance,
\begin{eqnarray}\label{last convolution measure}
\sum_{i\geq 1}Var(\langle Y,e_i\rangle)<1.
\end{eqnarray}

\subsection{Weak solutions of a class of stochastic differential equations}

Consider the system of stochastic differential equations
\begin{eqnarray}
\left\{ \begin{array}{ll}\label{SDE}
dX_t=b_1(Y_t)dt+dB_t^1, & X_0=x \\
dY_t=b_2(X_t)dt+dB_t^2, & Y_0=y
\end{array}\right.
\end{eqnarray}
where $\{(B_t^1,B_t^2)\}_{t\in [0,1]}$ is a two-dimensional standard Brownian motion defined on the probability space $(\Omega,\mathcal{F},\mathcal{P})$, $x,y\in\mathbb{R}$ and $b_1,b_2:\mathbb{R}\to\mathbb{R}$ are measurable functions. If for $i\in\{1,2\}$ the Novikov condition
\begin{eqnarray}\label{Novikov}
E\Big[\exp\Big\{\frac{1}{2}\int_0^1|b_i(B_t^i)|^2dt\Big\}\Big]<+\infty
\end{eqnarray}
is satisfied, then by means of the Girsanov theorem one can assert that the process $\{\tilde{B}_t\}_{t\in [0,1]}$, where $\tilde{B}_t=(\tilde{B}_t^1,\tilde{B}_t^2)$ and for $i\neq j\in\{1,2\}$
\begin{eqnarray*}
\tilde{B}_t^i=B_t^i-\int_0^tb_i(B_s^j)ds
\end{eqnarray*}
is a standard two-dimensional Brownian motion under the probability measure $d\mathcal{Q}:=\mathcal{E}d\mathcal{P}$ with
\begin{eqnarray*}
\mathcal{E}:=\exp\Big\{\int_0^1b_1(B^2_t)dB_t^1+\int_0^1b_2(B^1_t)dB_t^2-\frac{1}{2}\int_0^1|b_1(B^2_t)|^2+|b_2(B^1_t)|^2 dt\Big\}.
\end{eqnarray*}
As a consequence, the process $\{B_t\}_{t\in [0,1]}$ becomes a weak solution of the system (\ref{SDE}) with respect to the probability space $(\Omega,\mathcal{F},\mathcal{Q})$ and the noise
$\{\tilde{B}_t\}_{t\in [0,1]}$. \\
Suppose that we want to investigate the law of the translated Brownian motion $\{\tilde{B}_t\}_{t\in [0,1]}$ under the original measure $\mathcal{P}$. We consider its first component, i.e.
\begin{eqnarray}\label{def of X}
\tilde{B}_t^1=B_t^1-\int_0^tb_1(B_s^2)ds,\quad t\in [0,1],
\end{eqnarray}
and we observe that the independence of $B^1$ and $B^2$ implies that the law of the process $\{\tilde{B}^1_t\}_{t\in [0,1]}$ under the measure $\mathcal{P}$ is obtained by convolving the Wiener measure (which is the law of the process $\{B_t^1\}_{t\in [0,1]}$) with a probability measure supported on the Cameron-Martin space (which is the law of the process $\{-\int_0^tb_1(B_s^2)ds\}_{t\in [0,1]}$). This means that we can proceed the investigation via the general framework of convolution measures described in the previous subsection. \\
Let $W$ be the classical Wiener space $C_0([0,1];\mathbb{R})$ of continuous functions $w$ defined on the interval $[0,1]$ with values on $\mathbb{R}$ and such that $w(0)=0$; $H$ the Cameron-Martin space $H_0^1([0,1];\mathbb{R})$ of absolutely continuous functions with square integrable derivative; $\mu$ the classical Wiener measure on $W$.\\
Denote by $X$ the process $t\mapsto\tilde{B}_t^1$ and by $Z$ and $Y$ the processes $t\mapsto B_t^1$ and $t\mapsto -\int_0^tb_1(B_s^2)ds$, respectively. According to (\ref{def of X}), we have $X=Z+Y$ with $Z$ independent of $Y$; the law of $Z$ is $\mu$ while the law of $Y$, say $\nu$, is supported on $H$. We can write explicitly the density of the law of $X$ with respect to $\mu$ via formula (\ref{density convolution}):
\begin{eqnarray}\label{5}
\int_H\mathcal{E}(h)d\nu(h)&=&\int_H\exp\Big\{\int_0^1\dot{h}_tdw_t-\frac{1}{2}\int_0^1\dot{h}_t^2dt\Big\}d\nu(h)\nonumber\\
&=&\int_W\exp\Big\{-\int_0^1b_1(\tilde{w}_t)dw_t-\frac{1}{2}\int_0^1b_1(\tilde{w}_t)^2dt\Big\}d\mu(\tilde{w}).
\end{eqnarray}
In the second equality we utilized the fact that the measure $\nu$ is the image of (an independent copy of) $\mu$ through the map $t\mapsto -\int_0^tb_1(\tilde{w}_s)ds$. According to the previous subsection, the density in (\ref{5}) belongs to $\mathcal{L}^2(W,\mu)$ if condition (\ref{exponential integrability}) is fulfilled; in the present framework this is equivalent to
\begin{eqnarray*}
\int_H\exp\Big\{\frac{1}{2}\int_0^1b_1(\tilde{w}_t)^2dt\Big\}d\mu(\tilde{w})<+\infty,
\end{eqnarray*}
which is exactly the Novikov condition (\ref{Novikov}). Therefore, Assumption \ref{square integrable density} is equivalent to the Novikov condition (\ref{Novikov}). Furthermore, Assumption \ref{positive trace class} is satisfied according to the discussion of the previous subsection. Let us now focus on Assumption \ref{integrable quadratic exponential}. We know that inequality (\ref{last convolution measure}) is sufficient for that assumption to be true. Let $\{e_i\}_{i\geq 1}$ be an orthonormal bases of $H$; then
\begin{eqnarray*}
\sum_{i\geq 1}Var(\langle Y,e_i\rangle)&=&\sum_{i\geq 1}E[\langle Y,e_i\rangle^2]-(E[\langle Y,e_i\rangle])^2\\
&\leq&\sum_{i\geq 1}E[\langle Y,e_i\rangle^2]\\
&=&\sum_{i\geq 1}\int_W\Big(\int_0^1b_1(w_t)\dot{e}_i(t)dt\Big)^2d\mu(w)\\
&=&\int_W\sum_{i\geq 1}\Big(\int_0^1b_1(w_t)\dot{e}_i(t)dt\Big)^2d\mu(w)\\
&=&\int_W\int_0^1b^2_1(w_t)dtd\mu(w).
\end{eqnarray*}
To conclude, if
\begin{eqnarray*}
\int_0^1E[|b_1(B_t^2)|^2]dt<1
\end{eqnarray*}
then Assumption \ref{integrable quadratic exponential} is satisfied.

\section{Preliminary results}

We are now going to collect several important results of independent interest that will play a crucial role in the proof of Theorem \ref{main theorem}. The next proposition establishes the existence in the Cameron-Martin space $H$ of the mean $E[X]$ of a random element $X$ on $W$.

\begin{proposition}\label{existence of mean}
Let $X$ be a random variable taking values on $W$. Assume that the law of $X$ is absolutely continuous with respect to $\mu$ with a density $f$ belonging to $\mathcal{L}^2(W,\mu)$. Then, the expectation $E[X]$ of $X$ belongs to $H$ and coincides with the first kernel in the chaos decomposition of $f$. Moreover, the density of $X-E[X]$ is given by $f\diamond\mathcal{E}(-E[X])$.
\end{proposition}

\begin{proof}
The expectation of $X$ is defined to be the unique element $E[X]\in W$ such that
\begin{eqnarray}\label{def expectation}
E[\langle X,\varphi\rangle]=\langle E[X],\varphi\rangle,\quad\mbox{ for all }\quad\varphi\in W^*.
\end{eqnarray}
Its existence is guaranteed by the finiteness all the moments of $\langle X,\varphi\rangle$, as explained in the first section. Moreover,
\begin{eqnarray*}
E[\langle X,\varphi\rangle]&=&\int_W\langle w,\varphi\rangle f(w)d\mu(w)\\
&=&\langle f_1,\varphi\rangle_H
\end{eqnarray*}
where we utilized equation (\ref{orthogonality chaos expansion}) below and the identity $\langle w,\varphi\rangle=\delta^1(\varphi)$. Comparing with equation (\ref{def expectation}), this shows that $E[X]=f_1\in H$. In addition, using identity (\ref{S-transform}) we can write
\begin{eqnarray*}
\int_We^{i\langle\cdot,\varphi\rangle}(f\diamond\mathcal{E}(-E[X]))d\mu&=&e^{-\frac{|\varphi|_H^2}{2}}
\int_W\mathcal{E}(i\varphi)(f\diamond\mathcal{E}(-E[X]))d\mu\\
&=&e^{-\frac{|\varphi|_H^2}{2}}\int_W\mathcal{E}(i\varphi)fd\mu
\int_W\mathcal{E}(i\varphi)\mathcal{E}(-E[X])d\mu\\
&=&e^{-\frac{|\varphi|_H^2}{2}-i\langle E[X],\varphi\rangle_H}\int_W\mathcal{E}(i\varphi)fd\mu\\
&=&e^{-i\langle E[X],\varphi\rangle_H}\int_We^{i\langle\cdot,\varphi\rangle}fd\mu\\
&=&e^{-i\langle E[X],\varphi\rangle_H}E\Big[e^{i\langle X,\varphi\rangle}\Big]\\
&=&E\Big[e^{i\langle X-E[X],\varphi\rangle}\Big].
\end{eqnarray*}
Here we employed the fact that $\langle E[X],\varphi\rangle_H=\langle E[X],\varphi\rangle$ since $E[X]\in H$ and $\varphi\in W^*$. The proof is complete.
\end{proof}

\noindent Since we are dealing with random elements taking values on possibly infinite dimensional abstract Wiener spaces, it is not clear whether we can find an operator acting on densities which replicates the role of the classic convolution product on Euclidean spaces. The next theorem tells that the Wick product fulfills precisely this requirement. Similar results for the Poisson and chi-square distributions can be found in \cite{LS 2013} and \cite{LSportelli 2012}. In the sequel $\Gamma(\lambda)$ for $\lambda\in [0,1]$ denotes the operator defined in (\ref{def Gamma}); it corresponds to the Ornstein-Uhlenbeck semigroup $T_t$ via the relation $\Gamma(e^{-t})=T_t$ for any $t\geq 0$.

\begin{theorem}\label{wick product convolution}
Let $X_1,...,X_n$ be independent random variables taking values on $W$ and denote by $\mu_{X_1},...,\mu_{X_n}$ the corresponding laws on $W$, respectively. Assume that the measures $\mu_{X_1},...,\mu_{X_n}$
are absolutely continuous with respect to $\mu$. Then, for any $\alpha_1,...\alpha_n\in [0,1]$ such that $\alpha_1 + \cdot\cdot\cdot+\alpha_n =1,$ we have
\begin{eqnarray}\label{gaussian wick}
\Gamma(\sqrt{\alpha_1})\frac{d\mu_{X_1}}{d\mu}\diamond\cdot\cdot\cdot\diamond\Gamma(\sqrt{\alpha_n})\frac{d\mu_{X_n}}{d\mu} & = & \frac{d\mu_{\sqrt{\alpha_1} X_1+\cdot\cdot\cdot+\sqrt{\alpha_n} X_n}}{d\mu},
\end{eqnarray}
where $\frac{dQ}{d\mu}$ denotes the Radon-Nikodym derivative of the
measure $Q$ with respect to the reference measure $\mu$ and $\mu_{\sqrt{\alpha_1} X_1+\cdot\cdot\cdot+\sqrt{\alpha_n} X_n}$ denotes the law of the random variable $\sqrt{\alpha_1} X_1+\cdot\cdot\cdot+\sqrt{\alpha_n} X_n$.
\end{theorem}

\begin{proof}
See Proposition 3.1 in \cite{LS 2016}.
\end{proof}

\noindent According to the previous result, the Wick product can be considered to be a Gaussian analogue of the classic convolution product. From this point of view, the following theorem corresponds to a sharp Young-type inequality in the Gaussian framework.

\begin{theorem}\label{Young inequality}
Let $\alpha_1,...,\alpha_n\in [0,1]$ be such that $\alpha_1+\cdot\cdot\cdot+\alpha_n = 1$ and let
$p_1,...,p_n,r \in [1$, $+\infty]$ satisfy the following condition
\begin{eqnarray*}
\frac{\alpha_1}{p_1-1}+\cdot\cdot\cdot+\frac{\alpha_n}{p_n-1}=\frac{1}{r-1}.
\end{eqnarray*}
If $f_i\in\mathcal{L}^{p_i}(W,\mu)$ for each $i=1,...,n$, then $\Gamma(\sqrt{\alpha_1})f_1 \diamond
\cdot\cdot\cdot\diamond\Gamma(\sqrt{\alpha_n})f_n\in\mathcal{L}^r(W,\mu)$. More precisely,
\begin{eqnarray}\label{gaussian young}
\parallel\Gamma(\sqrt{\alpha_1})f\diamond\cdot\cdot\cdot\diamond\Gamma(\sqrt{\alpha_n})f_n\parallel_r & \leq &
\parallel f_1\parallel_{p_1}\cdot\cdot\cdot\parallel f_n\parallel_{p_n}.
\end{eqnarray}
\end{theorem}

\begin{proof}
See Theorem 4.7 in \cite{DLS 2011}.
\end{proof}

\subsection{Proof of Proposition \ref{properties quadratic exp}}
We first prove that $\xi\in\mathcal{L}^2(W,\mu)$:
\begin{eqnarray*}
\Vert\xi\Vert^2_2&=&\Big\Vert\sum_{k\geq 0}\frac{\delta^{2k}(g_2^{\otimes k})}{k!}\Big\Vert_2^2\\
&=&\sum_{k\geq 0}\frac{(2k)!}{(k!)^2}|g_2^{\otimes k}|^2_{H^{\otimes 2k}}\\
&=&\sum_{k\geq 0}{2k \choose k}|g_2|^{2k}_{H^{\otimes 2}}\\
&=&\frac{1}{\sqrt{1-4|g_2|^{2}_{H^{\otimes 2}}}}
\end{eqnarray*}
provided that
\begin{eqnarray*}
|g_2|^{2}_{H^{\otimes 2}}<\frac{1}{4}.
\end{eqnarray*}
Here we used the generating function of the central binomial coefficients, i.e
\begin{eqnarray*}
\sum_{k\geq 0}{2k \choose k}x^k=\frac{1}{\sqrt{1-4x}},\quad\mbox{ for }x<\frac{1}{4}.
\end{eqnarray*}
On the other hand, as explained in the introduction, Assumption \ref{integrable quadratic exponential} serves precisely to guarantee that $|g_2|^{2}_{H^{\otimes 2}}<\frac{1}{4}$, entailing the square integrability of $\xi$.\\
We now prove that $\xi$ is non negative by means of a characterization of positivity proposed by Nualart and Zakai in \cite{Nualart Zakai}. According to that paper, we need to prove that the function
\begin{eqnarray}\label{def F}
h\in H\mapsto F(h)=\int_We^{i\langle w,h\rangle}\xi(w)d\mu(w)
\end{eqnarray}
is positive definite, i.e. for all $z_1,...,z_n\in\mathbb{C}$ and $h_1,...,h_n\in H$ one has
\begin{eqnarray*}
\sum_{j,l=1}^nz_jF(h_i-h_l)\bar{z}_l\geq 0.
\end{eqnarray*}
It is not difficult to see that
\begin{eqnarray*}
F(h)&=&\sum_{k\geq 0}\frac{(-1)^{k}}{k!}\langle g_2^{\otimes k},h^{\otimes 2k}\rangle_{H^{\otimes 2k}}\cdot e^{-\frac{|h|_H^2}{2}}\\
&=&\sum_{k\geq 0}\frac{(-1)^{k}}{k!}\langle g_2,h^{\otimes 2}\rangle_{H^{\otimes 2}}^k\cdot e^{-\frac{|h|_H^2}{2}}\\
&=&\exp\Big\{-\langle g_2,h^{\otimes 2}\rangle_{H^{\otimes 2}}-\frac{|h|_H^2}{2}\Big\}.
\end{eqnarray*}
Since by Assumption \ref{positive trace class} the kernel $g_2=f_2-\frac{f_1^{\otimes 2}}{2}$ is positive, symmetric and of trace class, the function $\exp\{-\langle g_2,h^{\otimes 2}\rangle_{H^{\otimes 2}}\}$ is the characteristic functional of a Gaussian measure on $H$ with covariance operator equal to $2g_2$ and hence is positive definite. Moreover, $\exp\Big\{-\frac{|h|_H^2}{2}\Big\}$ is the characteristic functional of the Wiener measure $\mu$, entailing its positive definiteness. This show that the function in (\ref{def F}) is positive definite being the product of two functions of this type. This in turn also proves that $\xi d\mu$ is a Gaussian measure on $W$ obtained convolving the Wiener measure with a Gaussian measure on $H$ (whose covariance operator is equal to $2g_2$).\\
We are now left with the proof of (\ref{WID}). Note that according to Theorem \ref{wick product convolution}, since the law of the $\mathcal{X}_i$'s has density $\xi$ with respect to $\mu$, then the law of $\frac{\mathcal{X}_1+\cdot\cdot\cdot+\mathcal{X}_n}{\sqrt{n}}$ has density $\Gamma(1/\sqrt{n})\xi\diamond\cdot\cdot\cdot\diamond\Gamma(1/\sqrt{n})\xi$ ($n$-times). Therefore, (\ref{WID}) is equivalent to
\begin{eqnarray*}
\underbrace{\Gamma(1/\sqrt{n})\xi\diamond\cdot\cdot\cdot\diamond\Gamma(1/\sqrt{n})\xi}_{\mbox{$n$-times}}=\xi.
\end{eqnarray*}
By definition of Wick product (see (\ref{def wick product}) below),
\begin{eqnarray*}
\Gamma(1/\sqrt{n})\xi\diamond\cdot\cdot\cdot\diamond\Gamma(1/\sqrt{n})\xi &=&\sum_{k\geq 0}\sum_{i_1+\cdot\cdot\cdot+i_n=k}\frac{n^{-i_1}\cdot\cdot\cdot n^{-i_n}}{i_1!\cdot\cdot\cdot i_n!}\delta^{2k}(g_2^{\otimes i_1}\hat{\otimes}\cdot\cdot\cdot\hat{\otimes} g_2^{\otimes i_n})\\
&=&\sum_{k\geq 0}\frac{n^{-k}}{k!}\delta^{2k}\Big(\sum_{i_1+\cdot\cdot\cdot+i_n=k}\frac{k!}{i_1!\cdot\cdot\cdot i_n!}g_2^{\otimes i_1}\hat{\otimes}\cdot\cdot\cdot\hat{\otimes} g_2^{\otimes i_n}\Big)\\
&=&\sum_{k\geq 0}\frac{n^{-k}}{k!}\delta^{2k}((\underbrace{g_2+\cdot\cdot\cdot+g_2}_{\mbox{$n$-times}})^{\otimes k})\\
&=&\sum_{k\geq 0}\frac{\delta^{2k}(g_2^{\otimes k})}{k!}\\
&=&\xi.
\end{eqnarray*}
The proof is complete

\section{Proof of Theorem \ref{main theorem}}

We are now ready to prove our local limit theorem. The proof will be essentially based on the algebraic and analytical properties of the Wick product and Ornstein-Uhlenbeck semigroup, through the results described in the previous section.\\
Let $f\in\mathcal{L}^2(W,\mu)$ denote the common density of the $X_n$'s with respect to the measure $\mu$. According to Proposition \ref{existence of mean}, the density of $X_n-E[X_n]$ is given by $f\diamond\mathcal{E}(-f_1)$, where $f_1\in H$ denotes the first kernel in the Wiener-It\^o chaos decomposition of $f$. To ease the notation, we set
\begin{eqnarray*}\label{def f tilde}
\tilde{f}:=f\diamond\mathcal{E}(-f_1).
\end{eqnarray*}
We remark that if $f=1+\delta^1(f_1)+\delta^2(f_2)+\cdot\cdot\cdot$, then by definition of Wick product (\ref{def wick product}) and stochastic exponential (\ref{stoch exp. decomposition}) we get
\begin{eqnarray}\label{X-E[X]}
f\diamond\mathcal{E}(-f_1)&=&1+\delta^1(f_1-f_1)+\delta^2\Big(f_2+f_1\otimes (-f_1)+\frac{(-f_1)^{\otimes 2}}{2}\Big)+\cdot\cdot\cdot\nonumber\\
&=&1+\delta^2\Big(f_2-\frac{f_1^{\otimes 2}}{2}\Big)+\cdot\cdot\cdot.
\end{eqnarray}
From Theorem \ref{wick product convolution} we know that the density of $\frac{X_1+\cdot\cdot\cdot+X_n-nE[X_1]}{\sqrt{n}}$ is given by
\begin{eqnarray*}
\Gamma(1/\sqrt{n})\tilde{f}\diamond\cdot\cdot\cdot\diamond\Gamma(1/\sqrt{n})\tilde{f}=(\Gamma(1/\sqrt{n})\tilde{f})^{\diamond n}
\end{eqnarray*}
where $g^{\diamond n}$ stands for $g\diamond\cdot\cdot\cdot\diamond g$ ($n$-times). Moreover, the density of the random variable in (\ref{limit}) can be written as
\begin{eqnarray*}
\Gamma(\sqrt{\alpha})[(\Gamma(1/\sqrt{n})\tilde{f})^{\diamond n}]\diamond\Gamma(\sqrt{1-\alpha})1=[\Gamma(\sqrt{\alpha}/\sqrt{n})\tilde{f}]^{\diamond n}.
\end{eqnarray*}
Here we utilized the functorial property (\ref{functor}) and the identity $\Gamma(\lambda)1=1$ (note that the density of $Z$ with respect to $\mu$ is one). Observe in addition that we can write without ambiguity the right hand side of the previous equation as $\Gamma(\sqrt{\alpha}/\sqrt{n})\tilde{f}^{\diamond n}$ (again as a consequence of the interplay between Ornstein-Uhlenbeck semigroup and Wick product). Analogously, since the density of $\mathcal{X}$ is $\xi$, we get that the density of $\sqrt{\alpha}\mathcal{X}+\sqrt{1-\alpha}Z$ is $\Gamma(\sqrt{\alpha})\xi$.\\
Our aim is to prove that
\begin{eqnarray*}
\lim_{n\to +\infty}\Vert \Gamma(\sqrt{\alpha}/\sqrt{n})\tilde{f}^{\diamond n}-\Gamma(\sqrt{\alpha})\xi\Vert_1=0.
\end{eqnarray*}
First of all, exploiting the associativity and distributivity of the Wick product, together with Proposition \ref{properties quadratic exp}, we can write
\begin{eqnarray*}
\Gamma(\sqrt{\alpha}/\sqrt{n})\tilde{f}^{\diamond
n}-\Gamma(\sqrt{\alpha})\xi&=&\Gamma(\sqrt{\alpha}/\sqrt{n})\tilde{f}^{\diamond
n}-\Gamma(\sqrt{\alpha}/\sqrt{n})\xi^{\diamond n}\\
&=&(\Gamma(\sqrt{\alpha}/\sqrt{n})\tilde{f}-\Gamma(\sqrt{\alpha}/\sqrt{n})\xi)\diamond\\
&&\diamond\sum_{j=0}^{n-1}
\Gamma(\sqrt{\alpha}/\sqrt{n})\tilde{f}^{\diamond
j}\diamond\Gamma(\sqrt{\alpha}/\sqrt{n})\xi^{\diamond n-1-j}.
\end{eqnarray*}
Now, we take the $\mathcal{L}^1(W,\mu)$-norm and we apply Theorem \ref{Young inequality} (actually we need only the $\mathcal{L}^1$-form of the inequality which was proven before in the paper \cite{LS 2010}) and the triangle inequality to obtain
\begin{eqnarray}\label{1}
\Vert\Gamma(\sqrt{\alpha}/\sqrt{n})\tilde{f}^{\diamond
n}-\Gamma(\sqrt{\alpha})\xi\Vert_1&\leq&\Vert\Gamma(\sqrt{\alpha}/\sqrt{(1-\alpha)n})\tilde{f}-\Gamma(\sqrt{\alpha}/\sqrt{(1-\alpha) n})\xi\Vert_1\nonumber\\
&&\cdot\Big\Vert\sum_{j=0}^{n-1}
\Gamma(1/\sqrt{n})\tilde{f}^{\diamond
j}\diamond\Gamma(1/\sqrt{n})\xi^{\diamond n-1-j}\Big\Vert_1\nonumber\\
&\leq&\Vert\Gamma(\sqrt{\alpha}/\sqrt{(1-\alpha)n})\tilde{f}-\Gamma(\sqrt{\alpha}/\sqrt{(1-\alpha) n})\xi\Vert_1\nonumber\\
&&\cdot\sum_{j=0}^{n-1}
\Vert\Gamma(1/\sqrt{n})\tilde{f}^{\diamond
j}\diamond\Gamma(1/\sqrt{n})\xi^{\diamond n-1-j}\Vert_1.
\end{eqnarray}
Let us now focus the attention on the last sum. Invoking once again Theorem \ref{Young inequality} and exploiting the fact that $\tilde{f}$ and $\xi$ are density functions (their $\mathcal{L}^1(W,\mu)$-norms are equal to one) we get
\begin{eqnarray*}
&&\sum_{j=0}^{n-1}\Vert\Gamma(1/\sqrt{n})\tilde{f}^{\diamond
j}\diamond\Gamma(1/\sqrt{n})\xi^{\diamond n-1-j}\Vert_1\\
&\leq&\sum_{j=0}^{n-1}\Vert\Gamma(\sqrt{n-1}/\sqrt{n})\tilde{f}\Vert_1^{j}\Vert\Gamma(\sqrt{n-1}/\sqrt{n})\xi\Vert_1^{n-j-1}\\
&\leq&\sum_{j=0}^{n-1}\Vert \tilde{f}\Vert_1^{j}\Vert\xi\Vert_1^{n-j-1}\\
&=&n.
\end{eqnarray*}
Plugging this last estimate in (\ref{1}) we obtain
\begin{eqnarray}\label{2}
&&\Vert\Gamma(\sqrt{\alpha}/\sqrt{n})\tilde{f}^{\diamond
n}-\Gamma(\sqrt{\alpha})\xi\Vert_1\nonumber\\
&\leq&\Vert\Gamma(\sqrt{\alpha}/\sqrt{(1-\alpha)n})\tilde{f}-\Gamma(\sqrt{\alpha}/\sqrt{(1-\alpha) n})\xi\Vert_1\nonumber\\
&&\cdot\sum_{j=0}^{n-1}
\Vert\Gamma(1/\sqrt{n})\tilde{f}^{\diamond
j}\diamond\Gamma(1/\sqrt{n})\xi^{\diamond n-1-j}\Vert_1\nonumber\\
&\leq&n\Vert\Gamma(\sqrt{\alpha}/\sqrt{(1-\alpha)n})\tilde{f}-\Gamma(\sqrt{\alpha}/\sqrt{(1-\alpha) n})\xi\Vert_1
\end{eqnarray}
To ease the notation we set $\beta=\frac{\alpha}{1-\alpha}$ and we observe that, by the Nelson's hyper-contractive estimate \cite{Nelson} and the assumption $f\in\mathcal{L}^2(W,\mu)$, there exists $n_0\in\mathbb{N}$ big enough such that the function $\Gamma(\sqrt{\beta}/\sqrt{n_0})\tilde{f}$ belongs to $\mathcal{L}^2(W,\mu)$; therefore, we can write for all $n\geq n_0$ that
\begin{eqnarray*}
\Gamma(\sqrt{\beta}/\sqrt{n})\tilde{f}&=&\Gamma(\sqrt{n_0}/\sqrt{n})\Gamma(\sqrt{\beta}/\sqrt{n_0})\tilde{f}\\
&=&\sum_{k\geq 0}\Big(\frac{n_0}{n}\Big)^{\frac{k}{2}}\delta^k(\hat{f}_k)\\
&=&1+\frac{n_0}{n}\delta^2(\hat{f}_2)+\sum_{k\geq 3}\Big(\frac{n_0}{n}\Big)^{\frac{k}{2}}\delta^k(\hat{f}_k)
\end{eqnarray*}
where the $\hat{f}_k$'s are the kernels in the Wiener-It\^o chaos decomposition of $\Gamma(\sqrt{\beta}/\sqrt{n_0})\tilde{f}$.
The same holds true for $\Gamma(\sqrt{\beta}/\sqrt{n})\xi$, i.e.
\begin{eqnarray*}
\Gamma(\sqrt{\beta}/\sqrt{n})\xi&=&\Gamma(\sqrt{n_0}/\sqrt{n})\Gamma(\sqrt{\beta}/\sqrt{n_0})\xi\\
&=&\sum_{k\geq 0}\Big(\frac{n_0}{n}\Big)^{\frac{k}{2}}\delta^k(\hat{g}_k)\\
&=&1+\frac{n_0}{n}\delta^2(\hat{g}_2)+\sum_{k\geq 3}\Big(\frac{n_0}{n}\Big)^{\frac{k}{2}}\delta^k(\hat{g}_k).
\end{eqnarray*}
where the $\hat{g}_k$'s are the kernels of $\Gamma(\sqrt{\beta}/\sqrt{n_0})\xi$. Note that by construction $\hat{f}_2=\hat{g}_2$ (recall (\ref{X-E[X]}) and the definition of $g_2$ in Proposition \ref{properties quadratic exp}) which implies
\begin{eqnarray*}
\Gamma(\sqrt{\beta}/\sqrt{n})\tilde{f}-\Gamma(\sqrt{\beta}/\sqrt{n})\xi=\sum_{k\geq 3}\Big(\frac{n_0}{n}\Big)^{\frac{k}{2}}\delta^k(\hat{f}_k-\hat{g}_k)
\end{eqnarray*}
Hence, for all $n\geq n_0$
\begin{eqnarray*}
&&\Vert\Gamma(\sqrt{\beta}/\sqrt{n})\tilde{f}-\Gamma(\sqrt{\beta}/\sqrt{n})\xi)\Vert_1\\
&\leq&\Vert\Gamma(\sqrt{\beta}/\sqrt{n})\tilde{f}-\Gamma(\sqrt{\beta}/\sqrt{n})\xi)\Vert_2\\
&=&\Vert\Gamma(\sqrt{n_0}/\sqrt{n})\Gamma(\sqrt{\beta}/\sqrt{n_0})\tilde{f}-\Gamma(\sqrt{n_0}/\sqrt{n})\Gamma(\sqrt{\beta}/\sqrt{n})\xi)\Vert_2\\
&=&\Big(\sum_{k\geq 3}k!\Big(\frac{n_0}{n}\Big)^k|\hat{f}_k-\hat{g}_k|^2\Big)^{\frac{1}{2}}\\
&\leq&\Big(\frac{n_0}{n}\Big)^{\frac{3}{2}}\Big(\sum_{k\geq 3}k!|\hat{f}_k-\hat{g}_k|^2\Big)^{\frac{1}{2}}.
\end{eqnarray*}
Combining the last estimate with (\ref{2}) we conclude that
\begin{eqnarray*}
&&\Vert\Gamma(\sqrt{\alpha}/\sqrt{n})\tilde{f}^{\diamond n}-\Gamma(\sqrt{\alpha})\xi\Vert_1\\
&\leq&n\Vert\Gamma(\sqrt{\alpha}/\sqrt{(1-\alpha)n})\tilde{f}-\Gamma(\sqrt{\alpha}/\sqrt{(1-\alpha) n})\xi\Vert_1\\
&\leq&n\Big(\frac{n_0}{n}\Big)^{\frac{3}{2}}\Big(\sum_{k\geq 3}k!|\hat{f}_k-\hat{g}_k|^2\Big)^{\frac{1}{2}}\\
&=&\frac{C}{\sqrt{n}}\Big(\sum_{k\geq 3}k!|\hat{f}_k-\hat{g}_k|^2\Big)^{\frac{1}{2}}.
\end{eqnarray*}
The proof of the convergence is complete.

\section{Appendix}

In this section we recall for the reader's convenience few definitions and notations and collect some useful formulas that we utilized throughout the paper. For more details on the subject we refer the interested reader to one of the books \cite{Bogachev}, \cite{Janson}, \cite{Kuo Banach} and \cite{Nualart}.\\
For $f,g\in\mathcal{L}^2(W,\mu)$ with $f=\sum_{k\geq 0}\delta^k(f_k)$ and $g=\sum_{k\geq 0}\delta^k(g_k)$ one has the identity
\begin{eqnarray}\label{orthogonality chaos expansion}
\int_Wf(w)g(w)d\mu(w)=\sum_{k\geq 0}k!\langle f_k,g_k\rangle_{H^{\otimes k}}.
\end{eqnarray}
The \emph{stochastic exponential} is defined as
\begin{eqnarray}\label{def stochastic exponential}
w\in W\mapsto\mathcal{E}(h)(w):=\exp\Big\{\langle w,h\rangle-\frac{|h|_H^2}{2}\Big\},\quad h\in H.
\end{eqnarray}
Its chaos decomposition is given by
\begin{eqnarray}\label{stoch exp. decomposition}
\mathcal{E}(h)=\sum_{k\geq 0}\delta^k\Big(\frac{h^{\otimes k}}{k!}\Big).
\end{eqnarray}
For any $\lambda\in [0,1[$, we define the operator $\Gamma(\lambda)$ acting
on $\mathcal{L}^2(W,\mu)$ as
\begin{eqnarray}\label{def Gamma}
\Gamma(\lambda)\Big(\sum_{k\geq 0}\delta^k(f_k)\Big):=\sum_{k\geq 0}
\lambda^k\delta^k(f_k).
\end{eqnarray}
We observe that with $\lambda=e^{-\tau}$, $\tau\geq 0$, the operator $\Gamma(\lambda)$ coincides with the
Ornstein-Uhlenbeck semigroup
\begin{eqnarray*}
(P_{\tau}f)(w):=\int_Wf\big(e^{-\tau}w+\sqrt{1-e^{-2\tau}}\tilde{w}\big)d\mu(\tilde{w}),\quad
w\in W, \tau\geq 0,
\end{eqnarray*}
which is a contraction on $\mathcal{L}^p(W,\mu)$ for any $p\geq 1$. On the space $\mathcal{L}^2(W,\mu)$ one can define an unbounded multiplication between functions through the prescription
\begin{eqnarray}\label{def wick product}
\delta^k(f_k)\diamond\delta^j(f_j):=\delta^{k+j}(f_k\hat{\otimes} f_j)
\end{eqnarray}
where $\hat{\otimes}$ denotes the symmetric tensor product. This is named \emph{Wick product} of $\delta^k(f_k)$ and $\delta^j(f_j)$ and it is extended by linearity. It is easy to check that for $\lambda\in [0,1]$ and $f,g\in\mathcal{L}^2(W,\mu)$,
\begin{eqnarray}\label{functor}
\Gamma(\lambda)(f\diamond g)=\Gamma(\lambda)f\diamond\Gamma(\lambda)g
\end{eqnarray}
and for $h,l\in H$,
\begin{eqnarray*}
\mathcal{E}(h)\diamond\mathcal{E}(l)=\mathcal{E}(h+l)
\end{eqnarray*}
and
\begin{eqnarray}\label{S-transform}
\int_W (f\diamond g)\mathcal{E}(h)d\mu=\int_W f\mathcal{E}(h)d\mu\cdot\int_W g\mathcal{E}(h)d\mu.
\end{eqnarray}
For additional information on the Wick product (and its role in the theory of stochastic differential equations)  we refer to the book by Holden et al. \cite{HOUZ}, the paper \cite{DLS 2013} and the references quoted there.


\begin{thebibliography}{99}

\bibitem{AOU}
K. Aase, B. {\O}ksendal and J. Ub{\o}e, Using the Donsker delta function to compute hedging
strategies, \emph{Potential Analysis} \textbf{14} (2001) 351-374.

\bibitem{Barron}
A. R. Barron, Entropy and the central limit theorem, \emph{Annals of  Probability} \textbf{14} (1986) 336-342.

\bibitem{Bentkus}
V. Bentkus, On the dependence of the Berry-Esseen bound on dimension, \emph{J. Statist.
Plann. Inference} \textbf{113} (2003) 385-402.

\bibitem{Blozenis}
M. Bloznelis, A note on the multivariate local limit theorem, \emph{Statistics and Proba-
bility Letters} \textbf{59} (2002) 227-233.

\bibitem{Bogachev}
V. I. Bogachev, \emph{Gaussian Measures}, American Mathematical Society, Providence,
1998.

\bibitem{DLS 2011}
P. Da Pelo, A. Lanconelli and A. I. Stan, A H\"older-Young-Lieb inequality for norms of Gaussian Wick products, \emph{Inf. Dim. Anal. Quantum Prob. Related Topics} \textbf{14} (2011) 375-407.

\bibitem{DLS 2013}
P. Da Pelo, A. Lanconelli and A. I. Stan, An It\^o formula for a family of stochastic
integrals and related Wong-Zakai theorems, \emph{Stochastic Processes and their Appli-
cations} \textbf{123} (2013) 3183-3200.

\bibitem{DLS 2016}
P. Da Pelo, A. Lanconelli and A. I. Stan, An extension of the Beckner's type Poincar\'e inequality to convolution measures on abstract Wiener spaces, \emph{Stochastic Analysis and Applications} \textbf{34} (2016) 47-64.

\bibitem{Davydov}
Y. Davydov, A variant of an infinite-dimensional local limit theorem, \emph{Journal of
Soviet Mathematics [1]} \textbf{61} (1992) 1853-1856.

\bibitem{Gnedenko}
B. V. Gnedenko, Local limit theorem for densities, \emph{Doklady Akad. Nauk SSSR} \textbf{95} (1954) 5-7.

\bibitem{HOUZ}
H. Holden, B. {\O}ksendal, J. Ub{\o}e and T.-S. Zhang, \emph{Stochastic Partial Differential Equations - II Edition}, Springer, New York, 2010.

\bibitem{Janson}
S. Janson, \emph{Gaussian Hilbert spaces}, Cambridge Tracts in
Mathematics \textbf{129}, Cambridge University Press, Cambridge, 1997.

\bibitem{Kuo Banach}
H. H. Kuo, \emph{Gaussian measures in Banach spaces}, Lecture Notes in Mathematics \textbf{463}, Springer, New York, 1975

\bibitem{L 2016}
A. Lanconelli, A new approach to Poincar\'e-type inequalities on the Wiener space, \emph{Stochastic and Dynamics} \textbf{16} (2016) 18 pages.

\bibitem{LSportelli 2012}
A. Lanconelli and L. Sportelli, Wick calculus for the square of a Gaussian random variable with application to Young and hypercontractive inequalities, \emph{Inf. Dim. Anal. Quantum Prob. Related Topics} \textbf{15} (2012) 16 pages.

\bibitem{LS 2010}
A. Lanconelli and A. I. Stan, Some norm inequalities for Gaussian
Wick Products, \emph{Stochastic Analysis and Applications} \textbf{28} (2010) 523-539.

\bibitem{LS 2013}
A. Lanconelli and A. I. Stan, A H\"older inequality for norms of Poissonian Wick products, \emph{Inf. Dim.
Anal. Quantum Prob. Related Topics} \textbf{16} (2013) 39 pages.

\bibitem{LS 2016}
A. Lanconelli and A. I. Stan, A note on a local limit theorem for Wiener space valued random variables, \emph{Bernoulli} \textbf{22} (2016) 2101-2112.

\bibitem{Lynnik}
Yu. V. Linnik, An information-theoretic proof of the central limit theorem with the
Lindberg condition, \emph{Theory of Probabability and Applications} \textbf{4} (1959) 288-299.

\bibitem{Nelson}
E. Nelson, The free Markoff field, \emph{Journal of Functional Analysis} \textbf{12} (1973) 211-227

\bibitem{Nualart}
D. Nualart, \emph{Malliavin calculus and Related Topics - II Edition}, Springer, New York,
2006.

\bibitem{Nualart Zakai}
D. Nualart and M. Zakai, Positive and strongly positive Wiener functionals, \emph{Barcelona Seminar on Stochastic Analysis} \textbf{32}, Birkh\"{a}user, Basel (1993) 132-146.

\bibitem{Prohorov}
Yu. V. Prohorov, On a local limit theorem for densities, \emph{Doklady Akad. Nauk SSSR} \textbf{83} (1952) 797-800.

\bibitem{Ranga Rao Varadarajan}
R. Ranga Rao and V. S. Varadarajan, A limit theorem for densities, \emph{Sankhya} \textbf{22} (1960) 261-266.

\end{thebibliography}
\end{document}